\documentclass[12pt,letterpaper]{article}
\pdfoutput=1
\usepackage{graphics,epsfig,graphicx,color}
\graphicspath{{../Figures/}{Figures/}}
\usepackage[caption = false]{subfig}
\usepackage{float}
\usepackage{capt-of}
\usepackage{amssymb,latexsym}

\usepackage{setspace}

\usepackage{amsmath}

\usepackage{mathrsfs,amsfonts}

\usepackage{amsthm,amsxtra}

\usepackage[final]{showkeys}

\usepackage{stmaryrd}

\usepackage{algorithm}
\usepackage{algorithmicx}
\usepackage{algpseudocode}

\usepackage[openbib]{currvita}
\usepackage{bibentry}
\usepackage{etaremune}
\usepackage{enumitem} 

\newif\ifPDF
\ifx\pdfoutput\undefined
	\PDFfalse
\else
	\ifnum\pdfoutput > 0
		\PDFtrue
	\else
		\PDFfalse
	\fi
\fi

\ifPDF
	\usepackage{pdftricks}
	\begin{psinputs}
		\usepackage{pstricks}
		\usepackage{pstcol}
		\usepackage{pst-plot}
		\usepackage{pst-tree}
		\usepackage{pst-eps}
		\usepackage{multido}
		\usepackage{pst-node}
		\usepackage{pst-eps}
	\end{psinputs}
\else
	\usepackage{pstricks}
\fi

\ifPDF
	\usepackage[debug,pdftex,colorlinks=true, 
	linkcolor=blue, bookmarksopen=false,
	plainpages=false,pdfpagelabels]{hyperref}
\else
	\usepackage[dvips]{hyperref}
\fi

\usepackage{fancyhdr}

\usepackage{comment}

\usepackage[toc,page]{appendix}

\usepackage{cancel}

\usepackage{multirow}

\newtheorem{theorem}{Theorem}[section]
\newtheorem{lemma}[theorem]{Lemma}
\newtheorem{definition}[theorem]{Definition}

\newcommand{\supp}{\operatorname{supp}}

\newcommand{\dsum}{\displaystyle\sum}

\newcommand{\dmin}{\displaystyle\min}

\newcommand{\eps}{\varepsilon}

\newcommand{\fc}{\mathfrak{c}}

\newcommand{\fu}{\mathfrak{u}}

\newcommand{\fii}{\mathfrak{i}}

\newcommand{\bbC}{\mathbb C}

\newcommand{\bbR}{\mathbb R}

\newcommand{\bnu}{{\boldsymbol \nu}}

 \newcommand{\bn}{\mathbf n}

 \newcommand{\bx}{\mathbf x} 
\newcommand{\by}{\mathbf y}

 \newcommand{\cB}{\mathcal B}
\newcommand{\cC}{\mathcal C}  
 
 \newcommand{\cH}{\mathcal H}

\newcommand{\cS}{\mathcal S} 
\newcommand{\cU}{\mathcal U}

\newcommand{\wt}{\widetilde}
\newcommand{\wh}{\widehat}

\newenvironment{keywords}
{\noindent{\bf Key words.}\small}{\par\vspace{1ex}}
\newenvironment{AMS}
{\noindent{\bf AMS subject classifications 2010.}\small}{\par}

\newcommand{\fs}{\mathfrak{s}}

\title{Imaging point sources in heterogeneous environments}

\author{
	Kui Ren\thanks{
		Department of Applied Physics and Applied Mathematics, Columbia University, New York, NY 10027;
		\href{mailto:kr2002@columbia.edu}{kr2002@columbia.edu}
	}
	\and
	Yimin Zhong\thanks{
		Department of Mathematics,
		University of California, Irvine, CA 92697-3875;
		\href{mailto:yiminz@uci.edu }{yiminz@uci.edu}
	}
}

\begin{document}
%%%%%%%%%%%%%%%%%%%%%%%%%%%%%%%%%%%%%%%%%%%%%%%%%%%%%%%%%%%%%%%%%%
%%%%%%%BEGIN DOCUMENT %%%%BEGIN DOCUMENT %%%%BEGIN DOCUMENT%%%%%%%
%%%%%%%%%%%%%%%%%%%%%%%%%%%%%%%%%%%%%%%%%%%%%%%%%%%%%%%%%%%%%%%%%%

\maketitle

%\tableofcontents

%%%%%%%%%%%%%%%%%%%%%%%%%%%
%%%%%%%%ABSTRACT%%%%%%%%%%%
%%%%%%%%%%%%%%%%%%%%%%%%%%%

\begin{abstract}
Imaging point sources in heterogeneous environments from boundary or far-field measurements has been extensively studied in the past. In most existing results, the environment, represented by the refractive index function in the model equation, is assumed \emph{known} in the imaging process. In this work, we investigate the impact of environment uncertainty on the reconstruction of point sources inside it. Following the techniques developed by El Badia and El Hajj (\emph{C. R. Acad. Sci. Paris, Ser. I, 350 (2012), 1031-1035}), we derive stability of reconstructing point sources in heterogeneous media with respect to measurement error as well as smooth changes in the environment, that is, the refractive index. Numerical simulations with synthetic data are presented to further explore the derived stability properties.
\end{abstract}

%%%%%%%%%%%%%%%%%%%%%%%%%%%
%%%%%%%%%KEYWORDS%%%%%%%%%%
%%%%%%%%%%%%%%%%%%%%%%%%%%%

\begin{keywords}
	Inverse source problems, Helmholtz equation, point sources, stability estimates, numerical reconstructions, uncertainty characterization, inverse problems
\end{keywords}

%%%%%%%%%%%%%%%%%%%%%%%%%%
%%%   AMS or PACS   %%%%%%
%%%%%%%%%%%%%%%%%%%%%%%%%%

\begin{AMS}
	15A29, 35R30, 49N45, 65N21, 78A46.
\end{AMS}

%%%%%%%%%%%%%%%%%%%%%%%%%%%%%%%%%%%%%%%%%%%%%%%%%%%%%%%%%%%%%%%%%%
%%%%%% BEGINNING TEXT %%% BEGINNING TEXT %%% BEGINNING TEXT %%%%%%
%%%%%%%%%%%%%%%%%%%%%%%%%%%%%%%%%%%%%%%%%%%%%%%%%%%%%%%%%%%%%%%%%%

%\RED{Maybe emphasize that all results in this paper is from one data set at a fixed frequency?}

%%%%%%%%%%%%%%%%%%%%%%%%%%%%%%%%%%%%%%%%%%%%%%%%%%%%%%%%%%%%%%%%%%
%%%%%%%%%%%%%%%%%%%%%%%%%%%%%%%%%%%%%%%%%%%%%%%%%%%%%%%%%%%%%%%%%%
\section{Introduction}
\label{SEC:Intro}
%%%%%%%%%%%%%%%%%%%%%%%%%%%%%%%%%%%%%%%%%%%%%%%%%%%%%%%%%%%%%%%%%%
%%%%%%%%%%%%%%%%%%%%%%%%%%%%%%%%%%%%%%%%%%%%%%%%%%%%%%%%%%%%%%%%%%

Recovering radiative sources inside heterogeneous media from boundary or far-field measurements has applications in many branches of science and technology~\cite{AmBaFl-SIAM02,Baltes-Book78,BaLiTr-JDE10,BaBeBeLe-IP05,CaEw-IPBVP75,Chin-ACM10,ElVa-IP09,GaOsTa-JIIPP13,HoCl-MB97,Ikehata-IP99,ImYa-IP98,IsLu-IPI18,JiLiYa-JDE17,Larsen-JQSRT73,MaTs-IP13,Nicaise-SIAM00,PuYa-IP96,Sezginer-IP87,Siewart-JQSRT93,Yamamoto-IP95}. Extensive mathematical and computational studies of such inverse source problems have been performed in the past decades; see, for instance, ~\cite{AnBuEr-Book97,Isakov-Book90,Isakov-Book06} and references therein for recent reviews on the subject. 

In this work, we are interested in a source recovery problem where the source to be reconstructed is the superposition of point sources~\cite{BaBeBeLe-IP05,CaEw-IPBVP75,ElBadia-IP05,ElHa-JIIPP02,ElEl-CRASP12,ElNa-IP11,FaHaEs-IPI13,GaZhCoWa-BOE10,KaLe-IP04,KoYa-IP02,LiTa-CiCP09,MaTs-IP13,OhInOh-IP11,Vessella-IP92}. Unlike general source functions, point sources are efficiently characterized by their locations and strengths, a fact that significantly reduces the dimension of the parameter space of the inverse problems. This dimension reduction often enables one to obtain uniqueness in the inverse problem with minimum amount of observed data and provides the possibility of utilizing efficient reconstruction algorithms, for instance these based on compressive sensing~\cite{ChMoPa-IP12,FaStYa-SIAM10}, in the source recovery process.

%There is a vast literature on point source reconstruction problems for under different physical and geometrical settings. ~\cite{}. Cannon and Ewing were the first ones to consider this type of problems~\cite{CaEw-IPBVP75}, The problem with point source are studied extensively by El Badia and collaborators. In~\cite{ElEl-CRASP12,ElHa-JIIPP02,ElBadia-IP05,ElHa-IP00,ElHa-IP01,ElBadia-IP05,ElHa-JIIPP02,ElNa-IP11}; Kang and Lee~\cite{KaLe-IP04}, Vessella~\cite{Vessella-IP92,OhInOh-IP11, KoYa-IP02,Vessella-IP92}, B. Farmer and C. Hall and S. Esedoglu~\cite{FaHaEs-IPI13}; L. Baratchart etal consider the problem in zero frequency case~\cite{BaBeBeLe-IP05}. It is also considered by El Badia~\cite{ElBadia-IP05}. Zero frequency problems with absorption has applications in bioluminesence tomography~\cite{GaZhCoWa-BOE10}; Ling and Takeuchi considered the point source reconstruction problems for the heat equation~\cite{LiTa-CiCP09}, Mamonov and Tsai considered point source identification in nonlinear advection-diffusion-reaction systems~\cite{MaTs-IP13}.

To formulate our problem, let $\Omega\subset \bbR^d$ ($d\ge 3$) be a simply connected domain with $\cC^2$ boundary $\partial\Omega$. Let $u(\bx)$ be the solution to the following boundary value problem to the Helmholtz equation:
\begin{equation}\label{EQ:Helmholtz}
\begin{array}{rcll}
\Delta u + k^2 \big(1+n(\bx)\big) u  &=& q(\bx), &\mbox{in}\ \ \Omega\\
	u &=& f(\bx), & \mbox{on}\  \partial\Omega
\end{array}
\end{equation}
where the real-valued function $n(\bx)$ is the refractive index, $q(\bx)$ and $f(\bx)$ are internal and boundary source terms respectively. We assume that $n(\bx)$ has a compact support in $\Omega$, that is, $\supp(n) \subset\subset \Omega$, and that $1+n(\bx)>0$, $\forall \bx\in\Omega$. We assume that $0$ is not an eigenvalue of the operator $\Delta + k^2(1+n)$ with homogeneous Dirichlet boundary condition such that the problem~\eqref{EQ:Helmholtz} admits a unique solution for given source functions $q$ and $f$.

We assume that the internal source function $q(\bx)$ is a superposition of $m$ point sources located at $\{\bx_j\}_{j=1}^m$ with strengths $\{\lambda_j\}_{j=1}^m$, that is,
\begin{equation}\label{EQ:Source}
	q(x) = \sum_{j=1}^m \lambda_j \delta(\bx-\bx_j).
\end{equation}
The strengths $\{\lambda_j\}$ are all assumed to be real-valued so that there is no physical absorption occurring at the point sources. 

The Helmholtz equation~\eqref{EQ:Helmholtz} can be viewed as a simplified frequency-domain model for either electromagnetic or ultrasound wave propagation, depending on the value of the parameters, mainly the wavenumber $k$, in the equation. The mathematical derivations in the rest of the paper implicitly assume that the wavenumber $k$ is real-valued and $k>0$. We believe, however, the same types of calculations can be carried out for the zero-frequency case ($k=0$) as those studied in ~\cite{BaBeBeLe-IP05,ElBadia-IP05} or in~\cite{WaLiJi-MP05} with an extra absorption term.

Let us also mention that since linearizing inverse coefficient problems often results in inverse source problems, the point source reconstruction problem we study in this paper is closely related to the problem of reconstructing small volume inclusions in background media~\cite{AmKa-IP03,AmMoVo-ESAIM03} and the problem of imaging small scatterers in complex media~\cite{BoPaTs-IP03,BoPaTs-IP05,ChMoPa-IP16,DeMaGr-JASA05}. The main difference is that in our derivation below, we can utilize explicitly the fact that point sources are singular. The secondary sources created by small scatterers or inclusions, however, do not carry the same level of singularity of  the point sources.

We are interested in the problem of reconstructing the point sources, i.e. their locations and strengths, from Cauchy data $(f, g)$ where the boundary measurement $g(\bx)$ is given as
\begin{equation}\label{EQ:Helmholtz Data}
	g(\bx)=\left( {\dfrac{\partial u}{\partial \nu}}\right){\bigg|}_{\partial\Omega} \equiv \bnu\cdot\nabla u|_{\partial\Omega},
\end{equation}
$\bnu(\bx)$ being the unit outer normal vector of the domain boundary at $\bx\in\partial\Omega$.

To our best knowledge, in all the previous work on point source recovery, the environment, that is, the refractive index $n(\bx)$ in our formulation, in which point sources (or point source like localized objects), are to be sought is assumed to be known exactly, with the only exception in ~\cite{BuLaTaTs-AFR09} where the authors tried to reconstruct point sources in homogeneous media with unknown impenetrable obstacles. In other words, in the environment under which the data are collected is the same as the environment that data are back-propagated to reconstruct the point sources. Moreover, the mathematical works where stability of source reconstructions are derived with respect to noise in measured data are all done under the assumption that the environment is homogeneous (so that one could have access to the explicit form of the associated Green's function).

In the rest of the paper, we prove a stability result, following the techniques developed by El Badia and El Hajj (\emph{C. R. Acad. Sci. Paris, Ser. I, 350 (2012), 1031-1035})~\cite{ElEl-CRASP12}, on the recovering of point sources in smooth \emph{inhomogeneous} environment, that is, when the refractive index $n(\bx)$ varies smoothly in space. We also prove the stability of point source reconstruction with respect to smooth changes in the environment: if we perform reconstructions in a medium that is only slightly, in appropriate sense, different than the medium from which we collected the data, then the reconstructions are only slightly different from the reconstructions in the exact medium that generated the data.

%The rest of the paper is structured as follows. We present in the next section the main results of this paper, the stability of the reconstruction with respect to noise in measured data and the stability of the reconstruction with respect to smooth environment change. We then present in Section~\ref{SEC:Num} numerical simulations based on synthetic data to demonstrate the stability we observed theoretically. Concluding remarks are offered in Section~\ref{SEC:Concl}.

%%%%%%%%%%%%%%%%%%%%%%%%%%%%%%%%%%%%%%%%%%%%%%%%%%%%%%%%%%%%%%%%%%
%%%%%%%%%%%%%%%%%%%%%%%%%%%%%%%%%%%%%%%%%%%%%%%%%%%%%%%%%%%%%%%%%%
\section{Main results}
\label{SEC:Results}
%%%%%%%%%%%%%%%%%%%%%%%%%%%%%%%%%%%%%%%%%%%%%%%%%%%%%%%%%%%%%%%%%%
%%%%%%%%%%%%%%%%%%%%%%%%%%%%%%%%%%%%%%%%%%%%%%%%%%%%%%%%%%%%%%%%%%

We now present the main result of this short paper: (i) the stability of reconstructing point sources in an heterogeneous medium; and (ii) the stability of the reconstructions with respect to smooth changes of the medium.

%%%%%%%%%%%%%%%%%%%%%%%%%%%%%%%%%%%%%%%%%%%%%%%%%%%%%%%%%%%%%%%%%%
\subsection{Stability in heterogeneous media}
\label{SUBSEC:Stab Hetero Media}
%%%%%%%%%%%%%%%%%%%%%%%%%%%%%%%%%%%%%%%%%%%%%%%%%%%%%%%%%%%%%%%%%%

We first consider the case where the underlying medium is known but heterogeneous, that is, we have a known but spatially varying refractive index $n(\bx)$. Under this circumstance, we can uniquely reconstruct the source $q(\bx)$ from a single pair of Cauchy data $(f, g)$ with a given H\"older type of stability. This is a generalization of the stability results of El Badia and El Hajj established in~\cite{ElEl-CRASP12,ElNa-IP11}.

We make the following general assumptions on the setup of the problem.

\noindent{\bf (A).} The domain $\Omega$ has a $\cC^2$ boundary $\partial\Omega$. The refractive index $n(\bx)$ is real-valued and smooth, with support $\supp(n)\subset\Omega$. The point sources are well separated in the sense that $\dmin_{i\neq j}|\bx_i-\bx_j|\ge \fc>0$ for some $\fc$. The point sources are sufficiently far away from the boundary of the domain such that ${\rm dist}(\bx_j,\partial\Omega)\ge \wt\fc>0$, $\forall j$, for some $\wt\fc$. The strengths of the point sources satisfy $0<\underline{\lambda} \le \lambda_j\le \overline{\lambda} <+\infty$, $\forall j$, for some $\underline{\lambda}$ and $ \overline{\lambda}$. The illumination boundary source $f$ is the restriction of a $\cC^\infty$ function to $\partial\Omega$.

It will be clear that the smoothness assumptions on the refractive index $n(\bx)$ and the boundary source $f$ are not completely necessary. In fact, being $\cC^3$ is sufficient for all the results to hold. It should also be noted that the assumptions on the point sources imply that $\bx_i\neq \bx_j$ whenever $i\neq j$, a fact that is implicitly used later when we study uniqueness of reconstructions.

%%%%%%%%%%%%%%%%%%%%%%%%%%%%%%%%%%%%%%%%%%%%%%%%
% LEMMA OF UNIQUENESS.
%
%%%%%%%%%%%%%%%%%%%%%%%%%%%%%%%%%%%%%%%%%%%%%%%%
\begin{lemma}\label{LEM:Uniqueness Known Media}
	Let 
	\begin{equation}\label{EQ:Source12}
			q_\ell = \sum_{j=1}^{m_\ell} \lambda_{\ell, j} \delta(\bx - \bx_{\ell, j}), \quad \ell=1,2
	\end{equation}
	be two sets of point sources satisfying the assumptions in~{\bf (A)}, $u_\ell$ ($\ell=1, 2$) the corresponding solutions to the Helmholtz equation~\eqref{EQ:Helmholtz}. Then $(f_1, g_1)=(f_2, g_2)$ implies that $m_1 = m_2$ $(:= m)$ and
	\[
		\quad \bx_{1,j} = \bx_{2,\pi(j)},\quad \lambda_{1,j} = \lambda_{2, \pi(j)},\ \ 1\le j\le m
	\]
	for some permutation $\pi \in {\rm Sym(m)}$.
\end{lemma}
\begin{proof}
	This result follows from the unique continuation principle for Cauchy problems of elliptic equations. Let $D_\eps^{\ell, j}:=\cB(\bx_{\ell, j}, \eps)$ be the disk centered at $\bx_{\ell, j}$ with radius $\eps > 0$ small enough such that $D_\eps^{\ell, j}\subset \Omega$. We define $w = u_1 - u_2$ and verify that $w$ solves
	\begin{equation}
		\begin{array}{cl}
		\Delta w + k^2 (1+n) w = 0, & \mbox{in}\ \Omega_{\eps} := \Omega\backslash \bigcup_{\ell, j} D_\eps^{\ell, j}\\
		\\
		w = 0, \quad \partial_{\nu} w = 0, & \mbox{on}\ \partial\Omega
		\end{array}
	\end{equation} 
	Since $\partial\Omega \subset \partial \Omega_{\eps}$, we then conclude from the unique continuation principle~\cite[Theorem 3.3.1]{Isakov-Book06} that $w(\bx) = 0$, $\forall \bx\in\Omega_{\eps}$ with any $\eps > 0$. If we take $\eps\to 0$, this implies that $w = 0$ except at the locations $\bx_{\ell, j}$. Therefore $w$ must be a finite linear combination of point sources and their derivatives. This is impossible. Therefore $q_1 = q_2$ up to a possible permutation $\pi$, that is, renumbering of the point sources.
\end{proof}

%%%%%%%%%%%%%%%%%%%%%%%%%%%%%%%%%%%%%%%%%%%%%%%%%%%%%%%%%%%%
% LEMMA FOR EXISTENCE OF LOCAL FRAME
%
%%%%%%%%%%%%%%%%%%%%%%%%%%%%%%%%%%%%%%%%%%%%%%%%%%%%%%%%%%%%%

We now study the stability of the reconstruction. Following~\cite{ElEl-CRASP12,ElNa-IP11}, we look at the stability issue for an algebraic reconstruction technique that is based on the projection of the point sources into planes in $\bbR^3$. Due to the fact that the medium is heterogeneous, we need to find good ways to do the projection. In the next two lemmas, we introduce our method of projection onto surfaces determined by the solutions of the Helmholtz equation (which are controlled by the medium).

\begin{lemma}\label{LEM:FRAME}
	Under the assumptions in~{\bf(A)} on the refractive index and $\Omega$, there exists a complex-valued function $\phi(\bx)$ and a constant $\mu>0$ such that $\phi(\bx)$ solves
	\begin{equation}\label{EQ:Schrodinger}
	\Delta \phi + k^2(1 + n) \phi = 0, \ \ \mbox{in}\ \ \Omega
	\end{equation}
	$|\nabla \phi(\bx)|\neq 0$, $\forall \bx\in\Omega$, and 
	\[
		\sup_{\bx\in\Omega}(|\phi(\bx)| + |\nabla \phi(\bx)|) < \mu \inf_{\bx\in\Omega} |\phi(\bx)|.
	\]
\end{lemma}
\begin{proof}
With the regularity of $n(\bx)$ assumed in~{\bf(A)}, we can take $\phi$ as the well-known complex geometrical optics (CGO) solution to~\eqref{EQ:Schrodinger}; see for instance~\cite{NaUhWa-IP13}. More precisely, let $\zeta = \eta + \fii\xi$ with $\eta\in\bbR^3$ and $\xi \in \bbR^3$ given vectors such that $\eta \cdot \xi = 0$ and $\zeta\cdot\zeta=k^2$ (i.e. $|\eta|^2-|\xi|^2=k^2$). It is shown in~\cite{NaUhWa-IP13} that~\eqref{EQ:Schrodinger} has a solution of the form
	\begin{equation*}
	\phi(\bx) = e^{i\zeta\cdot \bx} (1 + r(\bx)),\;\text{ with }\; \|r\|_{\cH^3(\Omega)} \le \fc \frac{k^2}{|\zeta|} \|n\|_{\cH^3(\Omega)}
	\end{equation*}
	when $|\zeta|$ is sufficiently large. Then by the Sobolev embedding theorem~\cite{AdFo-Book03}, $r\in \cC^{1,1/2}(\Omega)$. If we choose $|\zeta|$ large such that $\|r\|_{\cC^{1,1/2}(\Omega)} < 1/2$, then we will have $|\nabla \phi|\neq 0$ for all $\bx\in \Omega$ and $\phi\in \cC^{1,1/2}(\Omega)$. Therefore we can find a constant $\mu > 0$ such that
	\begin{equation}
	\sup_{\bx\in\Omega}(|\phi(\bx)| + |\nabla \phi(\bx)|) < \mu \inf_{\bx\in\Omega} |\phi(\bx)|.
	\end{equation}
	This completes the proof.
\end{proof}
%%%%%%%%%%%%%%%%%%%%%%%%%%%%%%%%%%%%%%%%%%%%%%%%%%%%%%%%
% DEFINE LOCAL FRAME
%
%%%%%%%%%%%%%%%%%%%%%%%%%%%%%%%%%%%%%%%%%%%%%%%%%%%%%%%%
\begin{definition}\label{DEF:PROJ}
The function $\phi$ introduced in Lemma~\ref{LEM:FRAME} defines a local frame $(\mathbf{e}_1,\mathbf{e}_2,\mathbf{e}_3)$ on $\Omega$, where $\mathbf{e}_1 = \frac{\nabla \phi}{|\nabla \phi|}\in \bbC^3$, and $[\mathbf{e}_1,\mathbf{e}_2,\mathbf{e}_3]$ forms an unitary matrix, which also determines a local coordinate change $\mathcal{U}: \mathbb{R}^3\rightarrow \mathbb{C}^3$, denoted by $\mathcal{U}(\bx) = (\fu_1(\bx), \fu_2(\bx), \fu_3(\bx))$, from the Cartesian coordinate to the local frame. We define a projection $\cS:\bbC^3\rightarrow \mathbb{C}$ as:
\begin{equation*}
\cS(\bx) = \overline{(\fu_2 + \fii \fu_3)}.
\end{equation*}
This projection defines a pseudo distance function $\text{\rm dist}_{\cS}(\bx, \by) = |\cS(\bx) - \cS(\by)|$ and 
the diameter of $\Omega$ under the projection $\cS$ is denoted by ${\rm diam}_{\cS}(\Omega) := \displaystyle\sup_{\bx, \by\in\Omega} {\rm dist}_{\cS}(\bx, \by)$.
\end{definition}
The project $\cS$ we defined here is clearly not unique in the sense that one can rotate the coordinates to use $(\fu_3,\fu_1)$ or $(\fu_1,\fu_2)$ to replace $\fu_2$ and $\fu_3$. Moreover, since $\phi$ can be chosen differently using the complex vector $\zeta$ (which controls the boundary condition needed), we can construct the $\cS$ that we need by selection a specific vector. 

The projection we just introduced allows us to construct the following test functions.
\begin{lemma}\label{LEM:Test Func}
Let $\{ \bx_j \}_{j=1}^m \subset \Omega$ be arbitrary distinct points. Then under assumption~{\bf(A)}, there exists a function $\psi(\bx)$ solving
\begin{equation}\label{EQ:Test Func}
	\Delta \psi + k^2(1+n) \psi = 0, \ \ \mbox{in}\ \ \Omega
\end{equation}
such that 
\begin{equation}\label{EQ:Test Func Zero}
	\psi(\bx_j) = 0, \qquad 1\le j\le m.
\end{equation}
\end{lemma}
\begin{proof}
We construct the function $\psi$ as follows:
\begin{equation}\label{EQ:Test Func Form}
\psi(\bx) = \phi(\bx) \prod_{j=1}^m \big(\cS(\bx)-\cS(\bx_j)\big).
\end{equation}
Then clearly $\psi(\bx_j) = 0$, $j=1, 2, \cdots, m$. It is straightforward to verify that $\Delta \cS = 0$ and $ \nabla \phi\cdot \nabla \cS = 0$, which allow us to check that $\psi$ solves ~\eqref{EQ:Test Func}.
\end{proof}

We are now ready to prove the stability of the reconstruction. Our reconstruction scheme follows a two-step process. In the first step, the locations of the point sources are probed by a projection method. In the second step, we use the reconstructed locations $\{\bx_j\}_{j=1}^m$ to reconstruct the strengths of the point sources $\{\lambda_j\}_{j=1}^m$.
\begin{theorem}\label{THM:Stab Data}
 Let $q_1$ and $q_2$ be two sources of the form~\eqref{EQ:Source12} with $m:=m_1=m_2$ that are reconstructed from the Cauchy data $(f_1, g_1)$ and $(f_2, g_2)$ respectively. Let $\mathcal{S}$ be the projection in Definition~\ref{DEF:PROJ}. Let $\sigma := \displaystyle\min_{\ell, i\neq j}|\cS(\bx_{\ell,i}) - \cS(\bx_{\ell ,j})|$ and assume that $\sigma>0$. Then, under the assumptions in {\bf(A)}, there exists a permutation $\pi\in{\rm Sym(m)}$ acting on $\{1, 2, \cdots, m\}$ and a constant $\fc_1$  depending on $\Omega$, $\phi$, $\cS$, and $m$, such that
	\begin{multline}\label{EQ:Stab Loc Data}
	\rho_{\bx}:=\max_{1\le j\le m} |\mathcal{S}(\bx_{1,j})-\mathcal{S}(\bx_{2,\pi(j)})| \\ 
	\le \fc_1 \left( \frac{\sqrt{|\partial\Omega|}({\rm diam}_{\cS}(\Omega))^{2m-1} }{\underline{\lambda} \sigma^{m-1}}\left( \|g_1 - g_2\|_{L^2(\partial\Omega)} + \|f_1 - f_2\|_{L^2(\partial\Omega)} \right)\right)^{\frac{1}{m}}
	\end{multline}
	where $|\partial\Omega|$ is the surface measure of $\partial\Omega$. Assume further that $\rho_{\bx}< \sigma$, and let $\wt\rho_x:=\displaystyle \max_{1\le j\le m}|\cU(x_{1,j}) -\cU(\bx_{2,\pi(j)})|$, then there exists constants $\fc_2$ and $\fc_3$, again depending on $\Omega$, $\phi$, $\cS$, and $m$, such that
	\begin{multline}\label{EQ:Stab Strength Data}
	\rho_\lambda:=\max_{1\le j\le m}|\lambda_{1,j} - \lambda_{2,\pi(j)}| \\ 
	\le  \fc_2 \overline{\lambda} \widetilde{\rho}_\bx   + \fc_3 \sqrt{|\partial\Omega|}({\rm diam}_{\cS}(\Omega))^{2m-2} \big(\|g_1 - g_2\|_{L^2(\partial\Omega)} + \|f_1 - f_2\|_{L^2(\partial\Omega)}\big).
	\end{multline}
	
\end{theorem}
\begin{proof}
	Let $u_\ell$ $(\ell=1, 2)$ be the solution to the Helmholtz equation~\eqref{EQ:Helmholtz} with source $q_\ell$. We define $w:=u_1-u_2$. Then $w$ solves the Helmholtz equation~\eqref{EQ:Helmholtz} with boundary data $(w, \partial_\nu w):=(f_1-f_2, g_1-g_2)$. Let $1\le j' \le m$ be an integer. From Lemma~\ref{LEM:Test Func}, we find a function, $\phi$ being defined in Lemma~\ref{LEM:FRAME}, 
	\begin{equation}\label{EQ:Psi j'}
		\psi_{j'}(\bx)=\phi(\bx)\prod_{i=1}^m(\cS(\bx) - \cS(\bx_{1,i})) \prod_{j=1, j\neq j'}^m (\cS(\bx) - \cS(\bx_{2,j}))
	\end{equation}
	that solves the equation
	\begin{equation}\label{EQ:Psi-t}
		\Delta \psi_{j'} + k^2(1 + n) \psi_{j'} = 0, \ \ \mbox{in}\ \ \Omega
	\end{equation}
	and satisfies
	\begin{equation}
		\psi_{j'}(\bx_{\ell,j}) = 0,\ \  \forall (\ell, j)\neq (2, j') .
	\end{equation} 
	
	Multiplying the equation for $w$ by $\psi_{j'}$ and the equation for $\psi_{j'}$ by $w$, taking the difference of the results, and applying Green's identity, we have, with $\fs(\bx)$ the Lebesgue measure on $\partial\Omega$,
	\begin{equation*}
	\begin{aligned}
	&\left| \int_{\partial\Omega} \big((f_1-f_2) \partial_{\nu} \psi_{j'} - \psi_{j'} (g_1 - g_2)\big)d\fs \right| =\left| \lambda_{2,j'} \psi_{j'}(\bx_{2,j'})\right|\\
	=& \left|\lambda_{2,j'} \phi(\bx) \prod_{i=1}^m(\cS(\bx_{2,j'}) - \cS(\bx_{1,i})) \prod_{j=1, j\neq j'}^m (\cS(\bx_{2,j'}) - \cS(\bx_{2,j})) \right|
	\ge  \theta \underline{\lambda} \rho_{j'}^m \sigma^{m-1}.
	\end{aligned}
	\end{equation*}
	where $\theta := \inf_{\bx\in\Omega}|\phi(\bx)|$,  $\rho_{j'} :=  \dmin_{1\le j\le m} |\cS(\bx_{2,j'})-\cS(\bx_{1,j})|$, and we have used $|\lambda_{2,j'}|\ge \underline{\lambda}$ by the assumptions in {\bf(A)}. 
	
	Meanwhile, by the Cauchy-Schwartz inequality, we have
	\begin{multline*}
	\left| \int_{\partial\Omega} \big((f_1-f_2) \partial_{\nu} \psi_{j'} - \psi_{j'} (g_1 - g_2)\big)d\fs \right| \\ 
	\le \|g_1 - g_2\|_{L^2(\partial\Omega)} \|\psi_{j'}\|_{L^2(\partial\Omega)} + \|f_1 -f_2\|_{L^2(\partial\Omega)} \|\partial_{\nu} \psi_{j'}\|_{L^2(\partial\Omega)},
	\end{multline*}
	where we can estimate 
	\begin{equation*}
	\begin{aligned}
	\|\psi_{j'}\|_{L^2(\partial\Omega)} &\le \vartheta \sqrt{|\partial\Omega|} ({\rm diam}_{\cS}(\Omega))^{2m-1},\\ \|\partial_{\nu}\psi_{j'}\|_{L^2(\partial\Omega)} &\le \vartheta  \sqrt{|\partial\Omega|}({\rm diam}_{\cS}(\Omega))^{2m-1} + \fc\theta \sqrt{|\partial\Omega|}({\rm diam}_{\cS}(\Omega))^{2m-2},
	\end{aligned}
	\end{equation*}
 with $\vartheta := \|\phi\|_{W^{1,\infty}(\Omega)}$ and $\fc =\fc(\Omega, \cS, m)$ a bounded constant. Using the fact that $\vartheta < \mu \theta$, given in Lemma~\ref{LEM:FRAME}, we conclude from the above calculations that
	\begin{equation}
 \underline{\lambda} \rho_{j'}^m \sigma^{m-1} \le  \wt\fc \sqrt{|\partial\Omega|}({\rm diam}_{\cS}(\Omega))^{2m-1} \big(\|g_1 - g_2\|_{L^2(\partial\Omega)} + \|f_1 - f_2\|_{L^2(\partial\Omega)}\big) 
	\end{equation}
	with $\wt\fc$ a bounded constant that depends on $\Omega$, $\phi$, $\cS$, and $m$. 
	
	Since $j'$ is taken arbitrarily, we conclude that
	\begin{equation}
	\max_{1\le j' \le m}  \rho_{j'} \le \left( \frac{\wt\fc \sqrt{|\partial\Omega|}({\rm diam}_{\cS}(\Omega))^{2m-1} }{\underline{\lambda} \sigma^{m-1}}\left( \|g_1 - g_2\|_{L^2(\partial\Omega)} + \|f_1 - f_2\|_{L^2(\partial\Omega)} \right)\right)^{1/m}.
	\end{equation}
	By the symmetry in our calculations between the two groups of point sources, we see that we could replace the left hand side of the above inequality with the Hausdorff distance between the two groups of projected points $\{\cS(\bx_{1,j})\}_{j=1}^m$ and $\{\cS(\bx_{2,j})\}_{j=1}^m$:
	\[
		\max\{ \max_{1\le j'\le m} \dmin_{1\le j\le m} |\cS(\bx_{2,j'})-\cS(\bx_{1,j})|, \max_{1\le j'\le m} \dmin_{1\le j\le m} |\cS(\bx_{1,j'})-\cS(\bx_{2,j})|\}.
	\]
	The stability result~\eqref{EQ:Stab Loc Data} then follows from this fact and the Hall theorem~\cite{Cameron-Book94,ElEl-CRASP12}, which states that there exists a permutation $\pi$ acting on $\{1,2,\cdots,m\}$, that is a renumbering of the points, such that the Hausdorff distance can be realized by $\rho_\bx$.
	 
	The next step is to establish the stability for the strengths of point sources. For an integer $1\le j'\le m$, we introduce the function
		\begin{equation*}
		\varphi_{j'}(\bx)=\phi(\bx)\prod_{j=1, j\neq j'}^m(\cS(\bx) - \cS(\bx_{1,j})) (\cS(\bx) - \cS(\bx_{2,\pi(j)})).
	\end{equation*}
	Then $\varphi_{j'}$ solves the equation
	\begin{equation*}
		\Delta \varphi_{j'} + k^2(1 + n) \varphi_{j'} = 0, \ \ \mbox{in}\ \ \Omega
	\end{equation*}
	and satisfies
	\begin{equation*}
		\varphi_{j'}(\bx_{\ell,j}) = 0,\ \  \forall (\ell, j) \notin \{(1, j'), (2, \pi(j'))\}.
	\end{equation*} 
	
		Following the same procedure as before, we multiply the equation for $w$ by $\varphi_{j'}$ and the equation for $\varphi_{j'}$ by $w$, take the difference of the results, and apply Green's identity to obtain,
	\begin{equation*}
	\left| \lambda_{2,\pi(j')} \varphi_{j'}(\bx_{2,\pi(j')}) - \lambda_{1,j'} \varphi_{j'}(\bx_{1,j'})\right|= \left| \int_{\partial\Omega} \left( (f_1 - f_2)\partial_{\nu}\varphi_{j'} - \varphi_{j'} (g_1 - g_2)  \right) d\fs \right|.
	\end{equation*}
	
	By the Cauchy-Schwartz inequality, we have
	\begin{multline*}
	\left| \int_{\partial\Omega} \big((f_1-f_2) \partial_{\nu} \varphi_{j'} - \varphi_{j'} (g_1 - g_2)\big)d\fs \right| \\ 
	\le \|g_1 - g_2\|_{L^2(\partial\Omega)} \|\varphi_{j'}\|_{L^2(\partial\Omega)} + \|f_1 -f_2\|_{L^2(\partial\Omega)} \|\partial_{\nu} \varphi_{j'}\|_{L^2(\partial\Omega)},
	\end{multline*}
	where
	\begin{equation*}
	\begin{aligned}
	\|\varphi_{j'}\|_{L^2(\partial\Omega)} &\le \vartheta \sqrt{|\partial\Omega|} ({\rm diam}_{\cS}(\Omega))^{2m-2},\\ \|\partial_{\nu}\varphi_{j'}\|_{L^2(\partial\Omega)} &\le \vartheta  \sqrt{|\partial\Omega|}({\rm diam}_{\cS}(\Omega))^{2m-2} + \fc\theta \sqrt{|\partial\Omega|}({\rm diam}_{\cS}(\Omega))^{2m-3},
	\end{aligned}
	\end{equation*}
 with $\vartheta$ and $\theta$ defined as before, and $\fc =\fc(\Omega, \cS, m)$ a bounded constant. We therefore have
	\begin{multline}\label{EQ:Stab Strength A}
	\left| \lambda_{2,\pi(j')} \varphi_{j'}(\bx_{2,\pi(j')}) - \lambda_{1,j'} \varphi_{j'}(\bx_{1,j'})\right| \\ 
	\le \wt\fc \sqrt{|\partial\Omega|}({\rm diam}_{\cS}(\Omega))^{2m-1} \big(\|g_1 - g_2\|_{L^2(\partial\Omega)} + \|f_1 - f_2\|_{L^2(\partial\Omega)}\big).
	\end{multline}
	
	We now verify, using the assumption that $\sigma > \bar \rho_\bx$, that,
	\begin{multline*}
	|\varphi_{j'}(\bx_{1,j'})|=|\phi(\bx_{1,j'})\prod_{j=1, j\neq j'}^m(\cS(\bx_{1,j'}) - \cS(\bx_{1,j})) (\cS(\bx_{1,j'}) - \cS(\bx_{2,\pi(j)}))|\\ \ge \theta \sigma^{m-1} \prod_{j=1, j\neq j'}^m \big| \left|\cS(\bx_{1,j'}) -\cS(\bx_{1,j})\right| - \left| \cS(\bx_{2,\pi(j)}) -\cS(\bx_{1,j})\right| \big| 
	\ge \sigma^{m-1} |\sigma -  \rho_\bx|^{m-1}.
	\end{multline*}
	This allows us to conclude that, for some constants $\fc'$ and $\fc''$ (for instance, one could take $\fc' = \theta^{-1} \sigma^{1-m}(\sigma -\rho_{\bx})^{1-m}$), we have
	\begin{equation}\label{EQ:Stab Strength B}
	\begin{aligned}
	\max_{1\le j'\le m} |\lambda_{2,\pi(j')}-\lambda_{1,j'}|&\le \fc' \max_{1\le j'\le m}|\lambda_{2,\pi(j')}\varphi_{j'}(\bx_{1,j'})-\lambda_{1,j'}\varphi_{j'}(\bx_{1,j'})| \\ 
	&\le \fc' \max_{1\le j'\le m}\lambda_{2,\pi(j')}|\varphi_{j'}(\bx_{1,j'})-\varphi_{j'}(\bx_{2,\pi(j')})| \\
	&\quad + \fc' \max_{1\le j'\le m}|\lambda_{2,\pi(j')}\varphi_{j'}(\bx_{2,\pi(j')})-\lambda_{1,j'}\varphi_{j'}(\bx_{1,j'})|\\ &\le \fc'' \overline{\lambda}\widetilde{\rho}_\bx	+ \fc' \max_{1\le j'\le m}|\lambda_{2,\pi(j')}\varphi_{j'}(\bx_{2,\pi(j')})-\lambda_{1,j'}\varphi_{j'}(\bx_{1,j'})|.
	\end{aligned}
\end{equation}	
The stability bound~\eqref{EQ:Stab Strength Data} then follows from ~\eqref{EQ:Stab Strength A} and~\eqref{EQ:Stab Strength B}.
\end{proof}

The stability of reconstructing the locations of the point sources is H\"older type with exponent $\frac{1}{m}$, $m$ being the number of point sources included. The stability deteriorates quickly when $m$ increases. Therefore, we could only hope to reconstruct stably a very small number of point sources in practice. 

The conditional stability of the reconstructing the strengths of the point sources contains two parts. The second part is from the Cauchy data and is Lipschitz type. The constant in front of it, however, depends on $m$. When ${\rm diam}_{\cS}(\Omega)$ is large, this constant blows up quickly with $m$, another indication that one can not hope to stably reconstruct a large number of point sources. The first part comes from error in the determination of the locations of the point sources. If the locations are reconstructed perfectly, this term disappear. If, on the other hand, there is a large error in the reconstructing of the locations, the error in the reconstruction of the strengths is also large.

%The main requirement is that we should take a projection $\cS$ such that the point sources remain separated in the projection, that is the quantity $\sigma$ in Theorem~\ref{THM:Stab Data} below is well-defined.

%%%%%%%%%%%%%%%%%%%%%%%%%%%%%%%%%%%%%%%%%%%%%%%%%%%%%%%%%%%%%%%%%%
\subsection{Stability with respect to media changes}
\label{SUBSEC:Uncertainty}
%%%%%%%%%%%%%%%%%%%%%%%%%%%%%%%%%%%%%%%%%%%%%%%%%%%%%%%%%%%%%%%%%%

Here we study the stability of the reconstruction of point sources with respect to smooth media changes. We assume that the measured data are collected with a medium $n_1$ that we do not know exactly. We then reconstruct the point sources pretending that the medium in which the data were collected is $n_2$. We show that the reconstructions in $n_2$ is not too different from the reconstructions in $n_1$ if $n_2$ is not too different from $n_1$, in appropriate sense.

\begin{theorem}\label{THM:Stab Medium}
	Let $q_1$ and $q_2$ be two sources of the form~\eqref{EQ:Source12} $(m:=m_1=m_2)$, reconstructed for two media with refractive index $n_1$ and $n_2$ respectively, using Cauchy data $(f, g)$. Under the assumptions in~{\bf(A)} for $(q_\ell, n_\ell)$ ($\ell=1,2$), there exists a permutation $\pi$ acts on $\{1,2\dots,m\}$ such that 
	\begin{equation}\label{EQ:Stab Medium}
	\max_{1\le j \le m} \left| \cS(\bx_{1,j}) - \cS(\bx_{2,\pi(j)})  \right| \le \fc  \left[ \frac{\mu k^2 }{\underline{\lambda} \sigma^{m-1}} \|n_1 - n_2\|_{L^{2}(\Omega)} \right]^{\frac{1}{m}},
	\end{equation}
$\sigma = \min_{i\neq j}|\cS(\bx_{2,i}) - \cS(\bx_{2,j})|$, $\mu$ is from Lemma~\ref{LEM:FRAME}, and $\fc=\fc(\Omega,m,\overline{\lambda}, g)$ is a bounded constant.
\end{theorem}
\begin{proof}
	Let $u_\ell$ $(\ell=1, 2)$ be the solution to the Helmholtz equation~\eqref{EQ:Helmholtz} with the source and refractive index pair $(n_\ell, q_\ell)$. We define $w:=u_1-u_2$, $\delta n=n_1-n_2$, and $\delta q=q_1-q_2$. Then $w$ solves 
	\begin{equation*}\label{EQ:Helmholtz Diff Media}
		\Delta w + k^2 (1+n_1) w = \delta q(\bx) -k^2 \delta n(\bx) u_2(\bx),\ \ \mbox{in}\ \ \Omega
	\end{equation*}
	with boundary data $(w, \partial_\nu w):=(0, 0)$. 
	
	Let $\phi$ be the function defined in Lemma~\ref{LEM:FRAME} for the medium with refractive index $n_1$. Let $1\le j' \le m$ be an integer. We use the function 	
	\begin{equation*}
		\psi_{j'}(\bx)=\phi(\bx)\prod_{i=1}^m(\cS(\bx) - \cS(\bx_{1,i})) \prod_{j=1, j\neq j'}^m (\cS(\bx) - \cS(\bx_{2,j})) .
	\end{equation*}
	This function now solves the equation
	\begin{equation*}
		\Delta \psi_{j'} + k^2(1 + n_1) \psi_{j'} = 0, \ \ \mbox{in}\ \ \Omega
	\end{equation*}
	and satisfies
	\begin{equation*}
		\psi_{j'}(\bx_{\ell,j}) = 0,\ \  \forall (\ell, j)\neq (2, j') .
	\end{equation*} 

	Multiplying the equation for $w$ by $\psi_{j'}$ and the equation for $\psi_{j'}$ by $w$, taking the difference of the results, and applying Green's identity, we have
	\begin{equation}
	k^2\int_{\Omega} (n_2 - n_1)(\bx) \psi_{j'}(\bx) u_2(\bx) d\bx = \lambda_{2,j'} \psi_{j'}(\bx_{2,j'})
	\end{equation}
	By the Cauchy-Schwartz inequality, we have
	\begin{equation*}
	k^2 \left| \int_{\Omega} (n_2 - n_1)(\bx) \psi_{j'}(\bx) u_2(\bx) d\bx \right| \le k^2 \|n_2 - n_1\|_{L^{2}(K)} \|\psi_{j'}\|_{L^{\infty}(\Omega)} \|u_2\|_{L^2(K)},
	\end{equation*}
	where $K := \supp(n_2 - n_1)\subset \subset \Omega$, $\|\psi_{j'} \|_{L^{\infty}(\Omega)} \le \|\phi\|_{L^{\infty}(\Omega)} ({\rm diam}_{\cS}(\Omega))^{2m-1}$. 
	
Let $G$ be the fundamental solution for $\Delta + k^2(1+n_2)$ with homogeneous Neumann boundary condition. We have the following representation for $u_2$:
	\begin{equation}\label{EQ:U2}
		u_2(\bx) = \sum_{j=1}^{m} \lambda_{2,j} G(\bx, \bx_{2,j}) - \int_{\partial\Omega} G(\bx, \by) g(\by) d\fs
	\end{equation}
	This allows us to conclude that $\|u_2\|_{L^2(K)}\le \fc(1+m \overline{\lambda})$ for some constant $\fc = \fc(\Omega, K, g)$.
	
	The rest of the proof is identical to that of Theorem~\ref{THM:Stab Data}. We first verify that
	\[
		\left| \lambda_{2,j} \psi_{j'}(\bx_{2,j'})\right| \ge  \theta \underline{\lambda} \rho_{j'}^{m} \sigma^{m-1}
	\]
	where $\rho_{j'}: =  \dmin_{1\le j\le m} {\rm dist}_{\cS}(\bx_{2,j'}, \bx_{1,j})$ is defined the same way as before.
	
	Therefore, we have obtained
	\begin{equation*}
	\max_{1\le j' \le m} \rho_{j'} \le \fc \left[ \frac{\mu k^2 }{\underline{\lambda} \sigma^{m-1}} \|n_1 - n_2\|_{L^2(\Omega)} \right]^{\frac{1}{m}},
	\end{equation*}
	The stability bounded~\eqref{EQ:Stab Medium} then follows from symmetry argument and the Hall theorem~\cite{Cameron-Book94,ElEl-CRASP12}.	
\end{proof}

This result shows that the reconstruction of the locations of the point sources, up to a permutation, is relatively robust against changes in the underlying medium. However, the stability again deteriorates fast when the number of point sources increases.

Note that the above result is based on the assumption that we know exactly the number of point sources inside the medium. We do not have a general uniqueness result that allows us to determine the number of point sources from the measurement in this case. However, in some special cases, we can hope to reconstruct uniquely the point sources in a unknown medium that is has sufficiently simple structures, utilizing the fact that point sources are more singular compared to media variations. 

\paragraph{Media with localized perturbations.} Let us consider the case where the medium $n_2$ is the homogeneous medium $n_1$ with a finite number of additional localized anomalies. That is, the refractive index $n_2$ is of the form:
\[
	n_2(\bx) = n_1+\sum_{k=1}^K \tau_k \chi_{\Omega_k},
\]
where $\Omega_k$ is the support of the $k$-th anomaly and $\tau_k$ is its strength. We assume that the point sources are away from the local anomalies of the medium, that is, ${\rm dist}(\bx_j, \Omega_k)\ge \fc>0$ $\forall j,k$, for some $\fc$. Then we could follow the same proof in Lemma~\ref{LEM:Uniqueness Known Media} to show that $w$ must vanish outside the support of $\delta n=n_2-n_1$. This allows us to uniquely (up to a permutation as before) determine the point sources that do NOT live in the support of the local anomalies of the medium. This provides a uniqueness argument for the numerical point source reconstructions in~\cite{BuLaTaTs-AFR09}.

%%%%%%%%%%%%%%%%%%%%%%%%%%%%%%%%%%%%%%%%%%%%%%%%%%%%%%%%%%%%%%%%%%
%%%%%%%%%%%%%%%%%%%%%%%%%%%%%%%%%%%%%%%%%%%%%%%%%%%%%%%%%%%%%%%%%%
\section{Numerical experiments}
\label{SEC:Num}
%%%%%%%%%%%%%%%%%%%%%%%%%%%%%%%%%%%%%%%%%%%%%%%%%%%%%%%%%%%%%%%%%%
%%%%%%%%%%%%%%%%%%%%%%%%%%%%%%%%%%%%%%%%%%%%%%%%%%%%%%%%%%%%%%%%%%

We now perform some numerical simulations in the context of the theoretical study in the previous section. We have two specific aims in mind: (i) when the underlying medium is known but heterogeneous, we want to see how well we can reconstruct point sources inside the medium; and (ii) when the underlying medium is not known, we want to see how the reconstruction of point sources are affected by the reconstruction of the medium. 

We therefore intend to reconstruct the medium as well as the point sources inside the medium. We assume that we have access to multiple Cauchy data sets. We use a two-step reconstruction process. In the first step, we form \emph{differential} data sets to eliminate the effect of the point sources and focus only on the medium. In the second step, we use a \emph{single} data set to reconstruct the point sources, so that our simulations are in consistent with the theory developed in the previous section.

\paragraph{Media reconstruction.} Let $u$ be the solution to the Helmholtz equation~\eqref{EQ:Helmholtz}, and $\wt u$ be the solution of the same equation but with boundary source $f(\bx)+h(\bx)$. Then we check that $v:=\wt u-u$ solves the Helmholtz equation
\begin{equation}\label{EQ:Helmholtz Diff Dif}
	\begin{array}{rcll}
		\Delta v + k^2\big(1+n(\bx)\big) v &=& 0,& \mbox{in}\ \ \Omega\\
		v & = & h(\bx), & \mbox{on}\ \ \partial\Omega
	\end{array}
\end{equation}
By changing the probe source $h(\bx)$, we could obtain data determined by the Dirichlet-to-Neumann operator
\begin{equation}\label{EQ:Data Diff}
	\Lambda_n:		h(\bx) \mapsto g(\bx)=\partial_\nu v_{|\partial\Omega}: =\partial_\nu \wt u_{|\partial\Omega}-\partial_\nu u_{|\partial\Omega} .
\end{equation}
These data allow us to reconstruct the refractive index $n(\bx)$ since that is the only unknown quantity in~\eqref{EQ:Helmholtz Diff Dif}. This inverse problem has been studied extensively; see for instance~\cite{Isakov-Book06,NaUhWa-IP13} and references therein.

We perform the reconstruction by reformulate the inverse problem as a minimization problem. Let us assume that we have data generated from $J$ different probe sources $\{h_j\}_{j=1}^J$. We reconstruct $n(\bx)$ by minimizing the following mismatch functional:
\begin{equation}
	\Phi(n):= \dfrac{1}{2}\sum_{j=1}^J \int_{\partial\Omega} (\Lambda_n h_j - g_j^*)^2 d\fs(\bx) + \frac{\beta}{2} \int_{\Omega} |\nabla n|^2 d\bx
\end{equation}
where $g^*_j$ is the measured differential data corresponding to the probe source $h_j$ and the parameter $\beta$ is the strength of the regularization term.

We solve this minimization problem with a quasi-Newton method~\cite{NoWr-Book06,ReBaHi-SIAM06} where we use the adjoint state method to calculate the gradient of the objective functional with respect to the refractive index.  Let $w_j$ $(1\le j\le J)$ be the solution to adjoint equation
\begin{equation}\label{EQ:Helmholtz Adj}
	\Delta w_j + k^2(1+n) w_j = 0,\ \ \mbox{in}\ \ \Omega, \qquad \ w_j=-(\Lambda_n h_j - g_j^*),\ \ \mbox{on}\ \ \partial\Omega
\end{equation}
We can then show that the Fr\'echet derivative of $\Phi$ with respect to $n$ in direction $\delta n$ is given as
\begin{equation}
	\Phi'(n)[\delta n] =k^2 \sum_{j=1}^J \int_\Omega w_j v_j \delta n(\bx) d\bx-\beta\Big[\int_{\Omega} (\Delta n) \delta n(\bx) d\bx-\int_{\partial\Omega}\partial_\nu n \delta n(\bx) d\fs(\bx)\Big].
\end{equation}

In the minimization process, we solve the forward and adjoint Helmholtz problems~\eqref{EQ:Helmholtz Diff Dif} and~\eqref{EQ:Helmholtz Adj} with a standard $P_1$ finite element solver. 

\paragraph{Source reconstruction.} Once the refractive index $n(\bx)$ is reconstructed, we can reconstruct the unknown point sources, encoded in $q(\bx)$ in the Helmholtz equation~\eqref{EQ:Helmholtz}, from observed boundary data. We do this again with a minimization strategy. More precisely, we minimize the functional
\begin{equation}\label{EQ:Obj Src}
		\Psi(\bx_1, \cdots, \bx_m, \lambda_1, \cdots, \lambda_m):= \frac{1}{2}\int_{\partial\Omega} (\bn\cdot\nabla u - g^{\ast})^2 d\fs(\bx) 
\end{equation}
over the locations and strengths of the point sources.

To avoid dealing with the singularity of the solution $u$ due to the point sources, we explicitly factorize out the singular part of $u$ as follows. Let $G(\bx;\by)$ be the fundamental solution of the homogeneous Helmholtz operator in the whole space, that is,
\[
	\Delta G + k^2 G = -\delta(\bx-\by), \ \ \mbox{in}\ \ \bbR^d.
\]
We represent the solution of~\eqref{EQ:Helmholtz} through the integral equation
\begin{multline}
	u(\bx)=-\sum_{j=1}^m \lambda_j G(\bx_j; \bx)+k^2\int_\Omega n(\by)u(\by)G(\by;\bx)d\by \\ 
	+\int_{\partial\Omega} G(\by;\bx) \bn\cdot \nabla u(\by) d\by-\int_{\partial\Omega} f(\by) \bn\cdot\nabla G(\by;\bx) d\by.
\end{multline}
Let $\wh u(\bx):=u(\bx)+\sum_{j=1}^m \lambda_j G(\bx_j; \bx)$, then $\wh u$ solves the integral equation
\begin{equation}\label{EQ:U Nonsingular}
	\wh u(\bx)=k^2\int_\Omega n(\by)\wh u(\by)G(\by;\bx)d\by +\int_{\partial\Omega} G(\by;\bx) \bn\cdot \nabla \wh u(\by) d\by+Q(\bx),
\end{equation}
where the source term
\begin{multline*}
	Q(\bx)=-k^2\sum_{j=1}^m \lambda_j\Big[ \int_\Omega n(\by)G(\bx_j;\by)G(\by;\bx)d\by +\int_{\partial\Omega} G(\by;\bx) \bn\cdot \nabla G(\bx_j;\by) d\by\Big] \\ -\int_{\partial\Omega} f(\by) \bn\cdot\nabla G(\by;\bx) d\by.
\end{multline*}
To find the solution $u$, we solve for $\wh u$ using~\eqref{EQ:U Nonsingular} and then form $u=\wh u- \sum_{j=1}^m \lambda_j G(\bx_j; \bx)$.

To evaluate the gradient of the objective function, we introduce the adjoint problem
\begin{equation}
	\Delta w + k^2(1+n) w = 0,\ \ \mbox{in}\ \ \Omega, \qquad \ w=\bn\cdot\nabla u-g^*,\ \ \mbox{on}\ \ \partial\Omega .
\end{equation}
We can then show that the gradient of $\Psi$ with respect to a parameter the strength $\lambda_k$ and location $\bx_k$ are given respectively as
\begin{equation}
	\dfrac{d\Psi}{d\lambda_k} =w(\bx_j),\qquad\mbox{and},\qquad 	\nabla_{\bx_k} \Psi=\lambda_j \nabla_\bx w(\bx_j).
\end{equation}

The numerical simulations we present below are all done in a two-dimensional domain for simplicity. The best way to make this consistent with the theory in the previous section, which are constructed in dimension three, is to view the the simulations as simplifications of three-dimensional ones for which the refractive index and the illumination sources are invariant in the $z$-direction. We set the domain $\Omega = [0, 1]^2$ and the wave number $k = 8$ in our experiments. We collect $J=6$ \emph{differential} data sets generated from sources $f$ and $\{h_j\}_{j=1}^{J}$ to reconstruct the refractive index. To avoid the inverse crime, the synthetic measurements are generated on a fine grid while the inversion is fulfilled on another coarse grid. Moreover, we pollute our synthetic data with multiplicative random noise by perform the operation:  $ g^*_j\to g^*_j(1 + \tau \cU([-1,1]))$ with $\cU[-1, 1]$ the uniformly distributed random variable in $[-1, 1]$ and $\tau$ the level of noise that we will specify later. The algorithms are implemented in the $\verb|MATLAB|$ software with the source codes deposited at $\verb|github|$~\footnote{The github repository for our source codes is at \href{https://github.com/lowrank/ips}{https://github.com/lowrank/ips}.}.

We performed simulations on several different media. Here we present results on two typical ones that have refractive indices respectively.
\begin{equation}\label{EQ:Med 1}
n(\bx) = 0.5 + 0.5\dsum_{k=1}^2\cos \left(\dfrac{\pi |\bx - \by_k|}{2R}\right)\chi_{D_R(\by_k)}
\end{equation}
where $\by_1 = (0.25, 0.25)$, $\by_2 = (0.75, 0.75)$, $R=0.25$ and $\chi_{D_R(\by_k)}$ is the characteristic function of the disk of radius $R$ centered at $\by_k$, and
\begin{equation}\label{EQ:Med 2}
	n(\bx)=0.4\chi_{Rec} + 0.2 \chi_{D_{0.2}(\by_3)},
\end{equation}
where $\chi_{Rec}$ is the characteristic function of the rectangle $Rec=(0.5\ 0.75)\times(0.25\ 0.75)$; see Figure~\ref{FIG:Unknown Media 1} and Figure~\ref{FIG:Unknown Media 2} respectively for the plots of these refractive indices.

\begin{figure}[htb!]
	\begin{center}
		\includegraphics[scale=0.5]{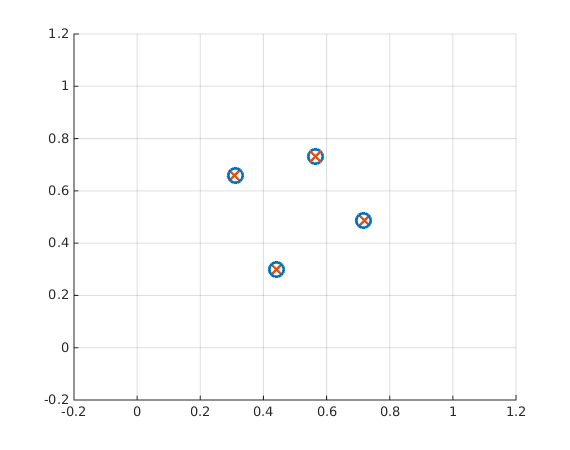}
		\includegraphics[scale=0.5]{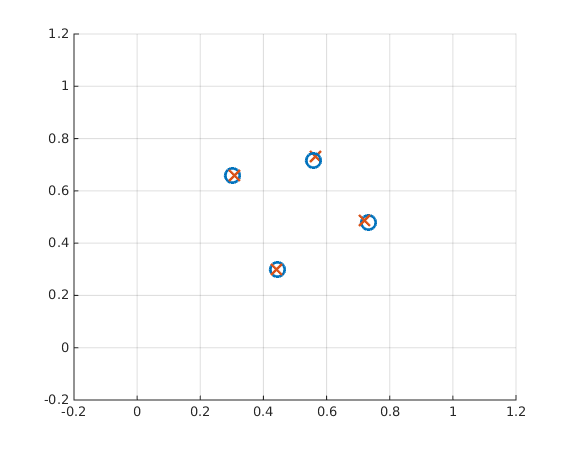}
\caption{Locations of the true (crosses: $\times$) and reconstructed (circles: $\circ$) point sources in medium~\eqref{EQ:Med 1} in Experiment 1. Shown are results with data contain $1\%$ (left) and $5\%$ (right) random noise respectively.}		
	\end{center}
\label{FIG:Known Med 1}
\end{figure}
\paragraph{Experiment 1 [Recovery in Known Environments].} In the first set of numerical experiments, we perform reconstructions of point sources in heterogeneous media with \emph{known} refractive indices. In Figure~\ref{FIG:Known Med 1} and Figure~\ref{FIG:Known Med 2}, we show reconstructions of the locations of the point sources in the media~\eqref{EQ:Med 1} and~\eqref{EQ:Med 2} respectively. The true and reconstructed strengths are summarized in the first two rows of Table~\ref{TAB:Intensities}. 
\begin{figure}[htb!]
	\begin{center}
\includegraphics[scale=0.5]{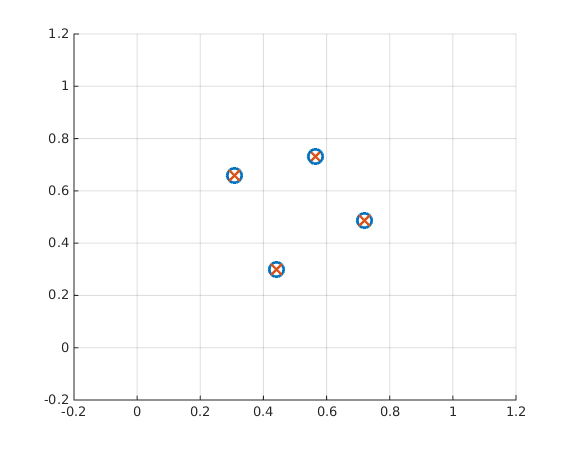}
\includegraphics[scale=0.5]{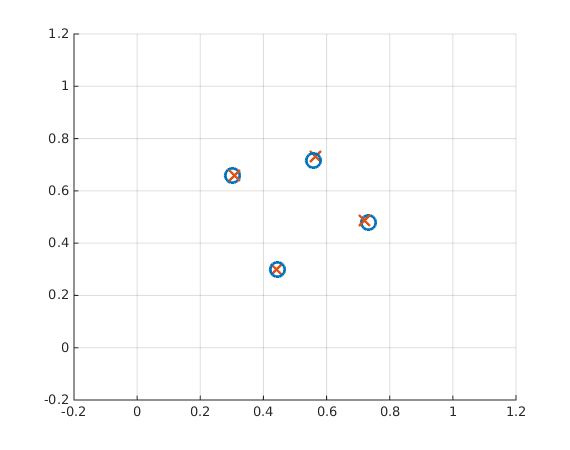}
\caption{Same as Figure~\ref{FIG:Known Med 1} but for the medium with refractive index~\eqref{EQ:Med 2}.}
	\end{center}
\label{FIG:Known Med 2}
\end{figure}
\begin{table}[hbt!]
	\centering
	\small	\caption{True and reconstructed intensities of the point sources, $(\lambda_1, \lambda_2, \lambda_3, \lambda_4)$, in different numerical experiments.}
	\begin{center}
	\begin{tabular}{lccc}
	\hline
	Experiment & True value & \multicolumn{2}{c}{Reconstructions with noisy data} \\
	& & 1\% noise & 5\% noise \\
	\hline
	%1: medium~\eqref{EQ:Med 1}& (0.885,0.725,0.705,0.523) & (0.906,0.730,0.677,0.543) & (0.921,0.647,0.761,0.457) \\
	%1: medium~\eqref{EQ:Med 2}& (0.885,0.725,0.705,0.523) & (0.883,0.724,0.709,0.522) & (0.881,0.717,0.687,0.538)\\
	%2 & (0.885,0.725,0.705,0.523) & (0.909,0.758,0.654,0.577) & (0.865,0.825,0.564,0.635) \\
	%3 & (0.885,0.725,0.705,0.523) & (0.895,0.806,0.645,0564) & (0.893,0.850,0.608,0.602)\\
	1: medium~\eqref{EQ:Med 1}& (0.89,0.73,0.71,0.52) & (0.91,0.73,0.68,0.54) & (0.92,0.65,0.76,0.46) \\
	1: medium~\eqref{EQ:Med 2}& (0.89,0.73,0.71,0.52) & (0.88,0.72,0.71,0.52) & (0.88,0.72,0.69,0.54)\\
	2 & (0.89,0.73,0.71,0.52) & (0.91,0.76,0.65,0.58) & (0.87,0.83,0.56,0.64) \\
	3 & (0.89,0.73,0.71,0.52) & (0.90,0.81,0.65,0.56) & (0.89,0.85,0.61,0.60)\\
	\hline
	\end{tabular}
	\end{center}
	\label{TAB:Intensities}
\end{table}
The simulations show that one can indeed reconstruct point sources, both their locations and their intensities, inside heterogeneous media when the media are not unreasonably complex. We want to emphasize here that in our theoretical analysis as well as numerical simulations, both the true sources and the sources to be reconstructed are explicitly assumed to be point sources. In other words, we explicitly search for the locations and strengths of the point sources, instead of reconstructing spatially distributed sources hoping that the result will give us point sources. In general, we observe from our extensive numerical simulations that when the refractive index is exactly known, the reconstructions are quite stable when the number of point sources is small. However, the reconstructions become too sensitive to algorithmic parameters when the number of point sources gets large.

We also want to emphasize that it is important to impose the constraints on the separability of the point sources in the numerical simulations. In other words, we have to explicitly ensure that the point sources to be reconstructed are far away from each other. Even in this case, the reconstructions are sensitive to the initial guess of the locations of the point sources. The objective function that we minimize to reconstruct the point sources can not differentiate between the true point sources and the equivalent class of re-labeled point sources. Therefore, the minimization algorithm could be easily fooled to jump between different intermediate configurations if the point sources are not well-separated. To further illustrate on this issue, we plot in Figure~\ref{FIG:Sentitivity Obj} the (normalized) objective functional $\Psi$ defined in~\eqref{EQ:Obj Src} as a function of the location of a single point source (the intensity of the source being assumed known). The true point source is located at  $(0.443, 0.298)$. While it is clear from the plot that the objective function is convex with respect to the location of the point source when $k=5$, this is not true anymore when $k=8$ and $k=12$. In the case of $k=8$, two local minmizers emerge at $y\approx 0.85$. More local minimizers emerge when $k=12$. These plots show that even in the case of a single point source, when the initial guess is far from the true position, the minimization algorithm could return wrong reconstructions. We can not visualize this phenomenon in the case of more than one point source. However, one can easily imagine that the situation would be far worse in that scenario.
\begin{figure}[htb!]
	\begin{center}
		\includegraphics[width=0.30\textwidth]{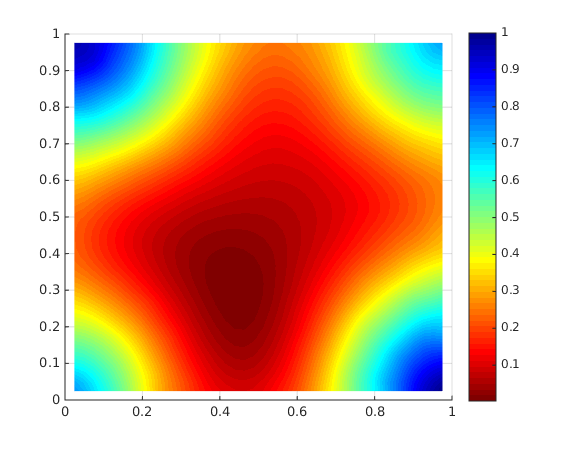}\,
		\includegraphics[width=0.33\textwidth]{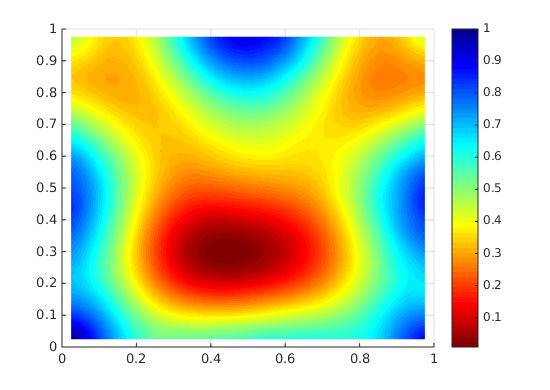}\,
		\includegraphics[width=0.315\textwidth]{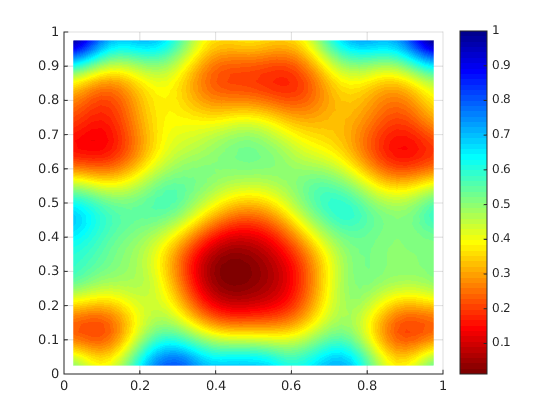}
\caption{Normalized objective function $\Psi$, defined in ~\eqref{EQ:Obj Src}, as a function of the location of a single point source. The true location is at  $(0.443, 0.298)$. Shown from left to right are the three cases of $k=5$, $k=8$ and $k=12$ respectively.}		
	\end{center}
\label{FIG:Sentitivity Obj}
\end{figure}

\paragraph{Experiment 2 [Recovery in an Unknown Environment].} In the second set of simulations, we reconstruct point sources in medium~\eqref{EQ:Med 1} assuming that both the medium and the point sources are not known. The reconstructions of the medium and the locations of the sources are shown in Figure~\ref{FIG:Unknown Media 1} and the reconstructed intensities of the point sources are summarized in the third row of Table~\ref{TAB:Intensities}.
\begin{figure}[!htb]
	\begin{center}
		\includegraphics[height=0.18\textwidth]{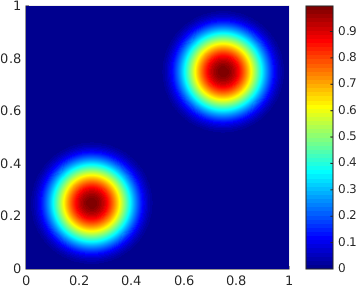}
		\includegraphics[height=0.18\textwidth]{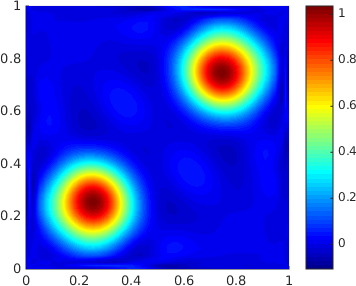}
		\includegraphics[height=0.18\textwidth]{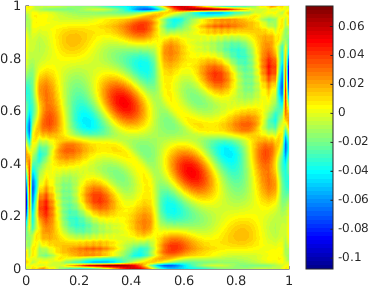}
		\includegraphics[height=0.18\textwidth]{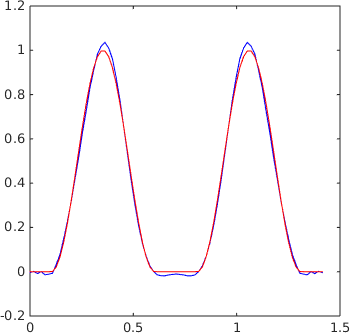}\\
		\includegraphics[scale=0.5]{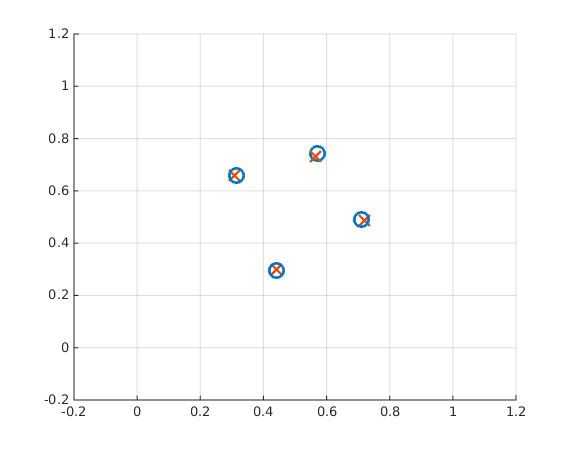}
		\includegraphics[scale=0.5]{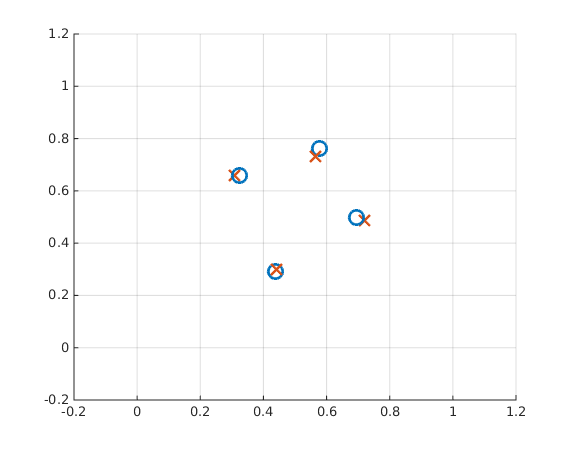}
	\end{center}
\caption{Simultaneous reconstruction of the point sources and the refractive index $n(\bx)$ in~\eqref{EQ:Med 1}. Top row: from left to right are the true $n(\bx)$, $n(\bx)$ reconstructed with data containing $1\%$ random noise, the difference between the true and the reconstructed $n(\bx)$, and the cross section of $n(\bx)$ along the diagonal (red and blue lines are for the true and the reconstruction respectively). Bottom row: true (crosses: $\times$) and reconstructed (circles: $\circ$) locations of the point sources using data with $1\%$ (left) and $5\%$ (right) random noise. }
\label{FIG:Unknown Media 1}
\end{figure}
Let us emphasize again that the reconstruction here is done in two steps. In the first step, we reconstruct the refractive index using multiple differential data sets. In the second step, we fix the refractive index, which is the reconstructed one, and reconstruct the point sources from \emph{one} Cauchy data set. If we compare the reconstructions in Figure~\ref{FIG:Unknown Media 1} with those in Figure~\ref{FIG:Known Med 1} (which are reconstructed under the true medium $n$ in~\eqref{EQ:Med 1}), and the reconstructions of intensities in the third row of Table.~\ref{TAB:Intensities} with those in the first row of the same table, we see that the reconstructions of the point sources are different but are of similar quality. That is,  smooth changes in the refractive index  introduces relatively small error in the reconstruction of point sources. This confirms our stability result in Theorem~\ref{THM:Stab Medium}.

\paragraph{Experiment 3 [Recovery in an Unknown Environment].} We repeat here the simulations in Experiment 2 for the medium~\eqref{EQ:Med 2}. The reconstructions of the medium and the locations of the sources are shown in Figure~\ref{FIG:Unknown Media 2} and the reconstructed intensities of the point sources are summarized in the last row of Table~\ref{TAB:Intensities}.
\begin{figure}[!htb]
	\begin{center}
	\includegraphics[height=0.18\textwidth]{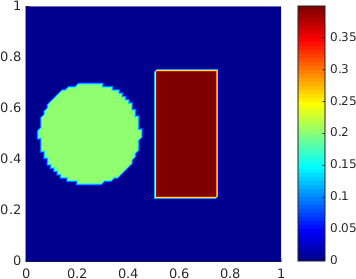}
	\includegraphics[height=0.18\textwidth]{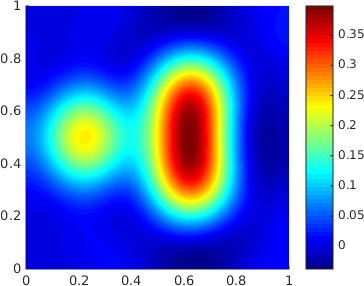}
	\includegraphics[height=0.18\textwidth]{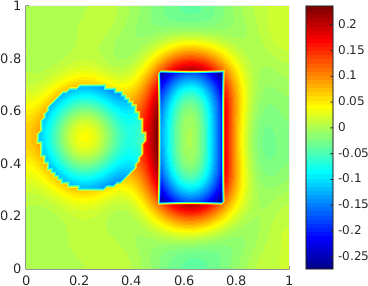}
	\includegraphics[height=0.18\textwidth]{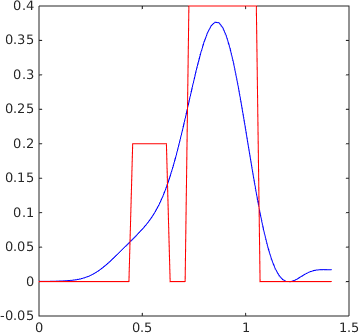}\\
	\includegraphics[scale=0.5]{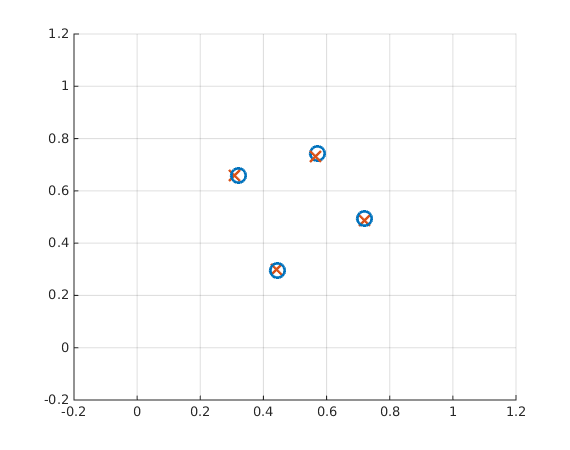}
	\includegraphics[scale=0.5]{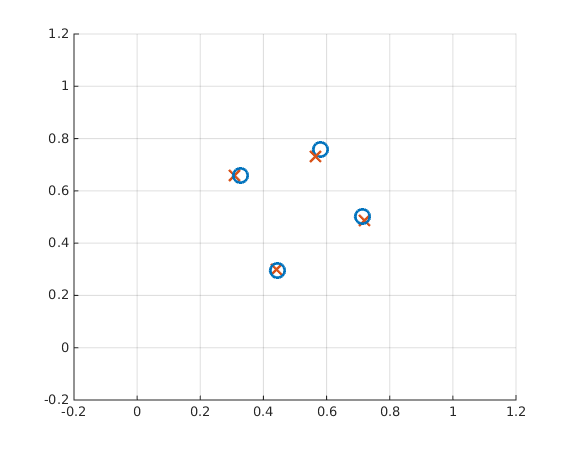}
	\end{center}
	\caption{Simultaneous reconstruction of the point sources and the refractive index $n(\bx)$ in~\eqref{EQ:Med 2}. Top row: from left to right are the true $n(\bx)$, $n(\bx)$ reconstructed with data containing $1\%$ random noise, the difference between the true and the reconstructed $n(\bx)$, and the cross section of $n(\bx)$ along the diagonal (red and blue lines are for the true and the reconstruction respectively). Bottom row: true (crosses: $\times$) and reconstructed (circles: $\circ$) locations of the point sources using data with $1\%$ (left) and $5\%$ (right) random noise.}
	\label{FIG:Unknown Media 2}
\end{figure}
The numerical results in this experiment again confirms the stability result in Theorem~\ref{THM:Stab Medium}. This can be seen by comparing the reconstructions in Figure~\ref{FIG:Unknown Media 2} with those in Figure~\ref{FIG:Known Med 2}, and the reconstructions of intensities in the fourth row of Table.~\ref{TAB:Intensities} with those in the second row of the same table.

The results in Experiment 2 and Experiment 3 demonstrate that smooth uncertainty in the medium produces relatively small errors in the reconstructions of the point sources. In other words, if we collect data from a medium that we know only approximately, we can simply perform reconstructions using our best known approximation to the medium. The results are not very different from those obtained using the true medium.

%%%%%%%%%%%%%%%%%%%%%%%%%%%%%%%%%%%%%%%%%%%%%%%%%%%%%%%%%%%%%%%%%%
%%%%%%%%%%%%%%%%%%%%%%%%%%%%%%%%%%%%%%%%%%%%%%%%%%%%%%%%%%%%%%%%%%
\section{Concluding remarks}
\label{SEC:Concl}
%%%%%%%%%%%%%%%%%%%%%%%%%%%%%%%%%%%%%%%%%%%%%%%%%%%%%%%%%%%%%%%%%%
%%%%%%%%%%%%%%%%%%%%%%%%%%%%%%%%%%%%%%%%%%%%%%%%%%%%%%%%%%%%%%%%%%

In this short paper, we studied, both theoretically and numerically, the reconstruction of point sources in heterogeneous media from boundary Cauchy data. 

Our first result is derived when the underlying medium is known. This is on the stability of the location and intensity reconstructions with respect to noise in the Cauchy data. This is a generalization of the results of El Badia and El Hajj in~\cite{ElEl-CRASP12} for the same reconstructions but in homogeneous media. Our numerical simulations confirm the theoretical predictions. More precisely, when only a very small number of point sources are to be reconstructed, numerical experiments suggest that they can be relatively stably recovered when the medium is known.

The motivation for our second result is to see how stable an imaging result, which could be the imaging of a point source as in our case, or a point scatter~\cite{AmKa-IP03,AmMoVo-ESAIM03}, or a reflector~\cite{FoGaPaSo-Book07}, or an extended target~\cite{BaCaLiRe-IP07,BaRe-SIAM08,Garnier-SIAM05,Zhao-SIAM04}, is with respect to uncertainties in the medium properties. This is an important problem to be addressed since in most applications, the underlying media are either assumed known or have to be reconstructed as well. In either case, targets are imaged with medium properties that are not the true medium properties. We established a stability result on the reconstruction of point sources with respect to smooth changes of the medium. This result says that if the medium is known up to a small (smooth) error, one can hope that the reconstructions are close to the true reconstructions. Numerical experiments show that even in the complicated case of simultaneous reconstructions of the refractive index and the point sources, the location of the sources can often be reconstructed in a robust way, indicating that the error caused by the uncertainty in the medium property, i.e. the refractive index, is relatively small.  More quantitative characterization of the uncertainty in the reconstructions needs to be performed, for instance, following the ideas presented in~\cite{ReVa-Prep18} in the context of photoacoustic imaging.

%%%%%%%%%%%%%%%%%%%%%%%%%%%%%%%%%%%%%%%%%%%%%%%%%%%%%%%%%%%%%%%%%%
%%%%%%%%%%%%%%%%%%%%%%%%%%%%%%%%%%%%%%%%%%%%%%%%%%%%%%%%%%%%%%%%%%
\section*{Acknowledgments}
%%%%%%%%%%%%%%%%%%%%%%%%%%%%%%%%%%%%%%%%%%%%%%%%%%%%%%%%%%%%%%%%%%
%%%%%%%%%%%%%%%%%%%%%%%%%%%%%%%%%%%%%%%%%%%%%%%%%%%%%%%%%%%%%%%%%%

We would like to thank Professor Abdellatif El Badia for useful discussion on algebraic methods for reconstructing point sources. This work is partially supported by the National Science Foundation through grant DMS-1620473.

%%%%%%%%%%%%%%%%%%%%%%%%%%%%%%%%%%%%%%%%%%%%%%%%%%%%%%%%%%%%%%%%%%
%%%%%%%%%%%%%%%%%%%%%%%%%%%%%%%%%%%%%%%%%%%%%%%%%%%%%%%%%%%%%%%%%%
{\small 
%\bibliography{C:/RenGDrive/Academic/Bibliography/RH-BIB}
%\bibliography{RH-BIB}
%\bibliographystyle{siam}

}
%%%%%%%%%%%%%%%%%%%%%%%%%%%%%%%%%%%%%%%%%%%%%%%%%%%%%%%%%%%%%%%%%%
%%%%%%%%%%%%%%%%%%%%%%%%%%%%%%%%%%%%%%%%%%%%%%%%%%%%%%%%%%%%%%%%%%

\end{document}